\documentclass{amsart}
\usepackage{amsmath,amsthm,mhequ,graphicx}
\renewcommand{\P}{\mathcal{P}}
\newcommand{\Q}{\mathcal{Q}}
\newcommand{\1}{\mathbf{1}}
\newcommand{\X}{\mathbf{X}}
\newcommand{\E}{\mathbf{E}}
\newcommand{\N}{\mathbf{N}}
\newcommand{\one}{\mathbf{1}}
\newcommand{\R}{\mathbf{R}}

\newcommand{\PP}{\mathbf{P}}
\newcommand{\ccdot}{\,\cdot\,}

\newtheorem{assumption}{Assumption}
\newtheorem{theorem}{Theorem}
\newtheorem{remark}[theorem]{Remark}
\newtheorem{lemma}[theorem]{Lemma}

\DeclareMathOperator{\cov}{cov}
\newcommand{\be}{\begin{equs}}
\newcommand{\ee}{\end{equs}}

\graphicspath{{./Graphics/}}

\newtheorem{proposition}[theorem]{Proposition}

\title{Coupling and Decoupling to bound an approximating Markov Chain}
\author[James E. Johndrow]{James E. Johndrow$^1$}
\thanks{$^1$Department  of Statistics, Stanford University, Stanford
  CA, 94305.}

\author[Jonathan C. Mattingly]{Jonathan C. Mattingly$^2$}
\thanks{$^2$Departments of Mathematics and Statistical Science, Duke University, Durham NC, 27701.}
\date{\today}

\begin{document}
\maketitle

This simple note lays out a few observations which are well known in
many ways but may not have been said in quite this way before. The
basic idea is that when comparing two different Markov chains it is
useful to couple them is such a way that they agree as often as
possible. We construct such a coupling and analyze it by
a simple dominating chain which registers if the two processes agree
or disagree. We find that this imagery is useful when thinking about
such problems. We are particularly interested in comparing the invariant measures
and long time averages of the processes. However, since the paths
agree for long runs, it also provides estimates on various stopping
times such as hitting or exit times.

This work builds on the general ideas of Maximal couplings. See for
instance \cite{MR1198659,MR1924231,MR1741181}. The analysis uses ideas from
Poisson equations and Martingale methods which we find
convient. These ideas are quite standard (see for instance
\cite{MR758799,MR1287609}) but we are more directly inspired by
\cite{Glynn_Meyn_1996,Kontoyiannis_Meyn_2011,MR2669996}.  As this work was heading to completion,
the authors became aware of the recent paper
\cite{2017arXiv170203917E} which also looks at coupling/decoupling of
two different Markov process for an interesting, but different goal.
Error bounds for approximations of uniformly mixing of Markov chains 
have been derived using other techniques by \cite{mitrophanov2005sensitivity},
among others.

In Section~\ref{Sec:Basic}, we give our basic setup  and quote some
simple,  classical convergence results. Proofs of some of the results
are given in the Appendix for completeness. In
Section~\ref{sec:nearby}, we introduce  a nearby Markov chain
and state the main approximation results of the note. 
In Section~\ref{sec:underst-thro-coupl}, we state the existence of a
coupling leading to certain desirable estimates. 
In Section~\ref{sec:usingCoup}, we show how the results of
Section~\ref{sec:nearby} can be obtained using  the  coupling from
Section~\ref{sec:underst-thro-coupl}. In
Section~\ref{sec:buildingTheCoupling}, we build the coupling on which
all results rest and introduce a two state change measure when the
chains are coupled or decoupled. In
Sections~\ref{sec:Analysis1} and \ref{sec:decoouling-time}, we analyse
the bounding chain. In Section~\ref{sec:Sharp}, we show certain
bounds are sharp. In Section~\ref{sec:MCMC}, we provide a simple
application to a Markov Chain Monte Carlo algorithm and show
numerically that the results of the paper show a good level of
approximation at considerable speed up by using an approximating chain
rather than the original sampling chain.

\vspace{1em}
\noindent{\bf Acknowledgements:}  This note grew of out of various
internal notes
prompted by  collaborations with Andrew Stuart, Mauro Maggioni, David Dunson, and Sayan
Mukherjee. We thank them for encouragement and useful discussions. We
are both indebted  to MSRI where during the Fall of 2015 we had the
time to largely write this version. We also thank the NSF for its suport
through grant DMS-1546130

\section{Basic ergodic statements}\label{Sec:Basic}
Let $\P$ be a Markov transition kernel on a Polish space $\X$ with
metric $|\ccdot|$. Given a function $\phi\colon \X \rightarrow \R$ and
probability measure $\nu$ on $\X$ we define:
\begin{align*}
  \P\phi(x) = \int_\X \phi(y)\P(x,dy)\,,\quad \nu \P(dy) = \int_\X
  P(x,dy)\nu(dx)\,,\quad \nu \phi = \int_\X \phi(y) \nu(dy)\,.
\end{align*}

The following assumption is a version of a Doeblin Condition.
\begin{assumption}\label{a:doeblin}
  There exists a constant $a\in (0,1)$ so that
  \begin{align*}
  \|  \P(x,\ccdot)-\P(y,\ccdot)\|_{TV} \leq 1-a 
  \end{align*}
for all $x,y \in \X$.
\end{assumption}

A standard result of such a Doeblin Condition is the following.
\begin{theorem}\label{thm:basicE}
  Under Assumption \ref{a:doeblin}, there exists a unique stationary
  measure $\mu$ for $\P$. Furthermore for any initial probability
  measures $\nu_1$ and $\nu_2$ one has
  \begin{align*}
    \|\nu_1 \P^n - \nu_2\P^n \|_{TV} \leq (1-a)^n \|\nu_1 - \nu_2\|_{TV}
  \end{align*}
\end{theorem}
Where for two probability measures $\nu_1$ and $\nu_2$, the Total
Variation distance is defined by
\begin{align*}
  \|\nu_1 - \nu_2\|_{TV} = \frac12 \Big(\sup_{|f|_\infty \leq 1} \nu_1 f -
  \nu_2 f \Big)= \inf_{(X_1,X_2)} \PP(X_1 \neq X_2)
\end{align*}
where $|f|_\infty=\sup_x|f(x)|$ and  the infimum is over all couplings of $\nu_1$ and $\nu_2$. In other 
words, $(X_1,X_2)$ are any random variables constructed on the same 
space with $\textrm{Law}(X_i)=\nu_i$.
It is equally straightforward to prove the law of large numbers and
concentration results. Defining 
\begin{align}\label{fstar}
  |f|_*= \inf_{\lambda \in \R} \Big(\sup_{x \in \X}|f(x) -\lambda|\Big)
\end{align}
as in \cite{MR2857021},
we have the following results whose proofs for completeness are given in Appendix~\ref{sec:proofs-doebl-results}.
\begin{theorem} For any bounded $f\colon \R \rightarrow \X$,  we have
  that
  \begin{align*}
    \E\Big (\frac1{n} \sum_{k=0}^{n-1} f(X_k) -\mu f\Big)^2 \leq
     \frac{4 |f|_*^2}{a^2n}\big( 2+\frac{8}{n}\big) 
  \end{align*}
and for any $\lambda >0$
\begin{align*}
    \PP \Big(\big|  \mu f - \frac1n \sum_{k=0}^{n-1} f(X_k) \big| \geq
  \frac{4}{n a} |f|_*+ \frac{\lambda}{\sqrt{n}}|f|_*    \Big)
  \leq 2\exp\big(- \frac{a^2 \lambda^2}{32}\big)
\end{align*}
\end{theorem}
It is worth noting that
  $|f|_* \leq \min ( |f|_\infty , |f- \mu f|_\infty )$ 
and hence either quantity on the right can replace $|f|_*$ in the
above estimates. 

\section{A nearby Markov chain}
\label{sec:nearby}

Now consider a second Markov chain $\P_\epsilon$. As the notation
suggests we are often  interested  in the setting when we have a
collection of Markov kernels $\{\P_\epsilon : \epsilon \in
(0,\epsilon_0]\}$ for some constant $\epsilon_0$. We want to
understand in what sense the long time dynamics of $\P_\epsilon$ are close to
those of $\P$. We begin with the following simple assumption.
\begin{assumption}\label{a:localApprox}
   There exists a constant $\epsilon>0$ so that 
  \begin{align*}
  \|  \P_\epsilon(x,\ccdot)-\P(x,\ccdot)\|_{TV} \leq \epsilon 
  \end{align*}
for     all $x \in\X$.
\end{assumption}

We have the following result
\begin{proposition}
  Under  Assumptions \ref{a:doeblin} and \ref{a:localApprox}, any
  stationary distribution $\mu_\epsilon$ of $\P_\epsilon$ satisfies
  \begin{align*}
    \|\mu - \mu_\epsilon\|_{TV} \leq \frac{\epsilon}{a}
  \end{align*}
\end{proposition}
\begin{proof}
  \begin{align*}
    \|\mu - \mu_\epsilon\|_{TV} &\leq   \|\mu\P -
    \mu_\epsilon\P\|_{TV}+  \|\mu_\epsilon\P -
    \mu_\epsilon\P_\epsilon\|_{TV}\leq (1-a)  \|\mu - \mu_\epsilon\|_{TV}+ \epsilon 
  \end{align*}
The first inequality follows from the triangle inequality; the second
used Assumption~\ref{a:doeblin} for the first term and
Assumption~\ref{a:localApprox} for the second term.  Rearranging the
resulting inequality produces the quoted result.
\end{proof}

\begin{proposition} Let  Assumptions \ref{a:doeblin} and
  \ref{a:localApprox} hold with $\epsilon \in (0,\frac{a}2)$. 
  Then  Assumption~\ref{a:doeblin} holds for the Markov operator
  $\P_\epsilon$ with the constant ``$a$'' equal to $a-2\epsilon$ which
  is less than 1 by construction. Hence for such $\epsilon$ the chain
  has a unique stationary distribution $\mu_\epsilon$ to which it
  converges exponentially. 
\end{proposition}


We now consider couplings of the chains $X_n,X_n^\epsilon$ evolving
according to the transition kernels $\P$ and $\P_\epsilon$ respectively
with $X_0$ and $X_0^\epsilon$ as initial conditions.

\begin{theorem}\label{cor:invAvg} Assume that Assumption~\ref{a:localApprox}  and 
  Assumption~\ref{a:crossDoeblin} hold. Then for any two probability 
  measures $\nu_1$ and $\nu_2$ on $\X$
  \begin{align*}
   \Big\|\frac1n \sum_{k=0}^{n-1} \nu_1\P^k -
   \frac1n \sum_{k=0}^{n-1}\nu_2\P_\epsilon^k\Big\|_{TV} \leq 
   \frac{\epsilon}{\alpha+\epsilon} +\frac{1-(1-\alpha-\epsilon)^n}{n(\alpha+\epsilon)}
    \Big(\|\nu_1-\nu_2\|_{TV} -\frac\epsilon{\alpha+\epsilon}\Big)\,. 
  \end{align*}
Furthermore, there exists a coupling of the process
$(X_n,X_n^\epsilon)$ with $\text{Law}(X_0)= \nu_1$ and
$\text{Law}(X_0^\epsilon)= \nu_2$ and a random
constant $K=K(\alpha,\epsilon,X_0,X_0^\epsilon)$ so 
\begin{align*}
  \frac12 \Big|\frac1n \sum_{k=0}^{n-1}    f(X_k)
     - \frac1n \sum_{k=0}^{n-1}    f(X_k^\epsilon)
    \Big| \leq \frac{\epsilon}{\alpha+\epsilon}+\frac{K}{n}+ 2\sqrt{\frac{\log(n)}{n}}
\end{align*}
for all $n >0$ and 
  \begin{align*}
     \E\Big(\frac1n \sum_{k=0}^{n-1}    f(X_k)
     - \frac1n \sum_{k=0}^{n-1}    f(X_k^\epsilon)
    \Big)^2 \leq  4|f|_*^2\Big(  \frac{\epsilon^2}{(\alpha+\epsilon)^2} +  \frac2{n^2}\frac1{(\alpha+\epsilon)^2} + \frac2n \frac{\alpha^2}{(\alpha+\epsilon)^4}\Big)
  \end{align*}
and for all $\lambda>0$,
\begin{align*}
  \PP\Big(\big|\frac1n \sum_{k=0}^{n-1}    f(X_k)- \frac1n \sum_{k=0}^{n-1}    f(X_k^\epsilon)  \big|
  \geq 2|f|_*\big( \frac\epsilon{\alpha+\epsilon} +
  \frac{\one_{\{X_0\neq X_0^\epsilon\}}}{n(\alpha+\epsilon)} + \frac{\lambda}{\sqrt{n}}
 \big) \Big)\leq e^{ - \tfrac{(\alpha+\epsilon)^2}{2}\lambda^2}\, .
\end{align*}
\end{theorem}
\begin{remark}
  Taking $\nu_1$ equal to $\mu$, the invariant measure of $\P$, will 
  produce estimates involving $\mu$, $\mu f$ and related quantities 
  more resembling Theorem~\ref{thm:basicE}. Alternative derivations of 
  estimates with this form can be found at the end of the Appendix in 
  Remark~\ref{rem: Closeness}. The disadvantage of this alternative 
  presentation is that it does not apply as directly to studying exit 
  times and other more pathwise variables as covered by the
  next result, Theorem~\ref{thm:Path}.
\end{remark}

Let $g$ be a real valued 
  function of a trajectory in $\X^\N$, which is the space of one-sided 
  infinite sequences $\{ X=(X_0,X_1,X_2,\dots) : X_k \in \X, k \in \N\}$. We will write $X$ and $X^\epsilon$ for 
  the entire trajectories. 

\begin{theorem}\label{thm:Path}
  Assume that Assumption~\ref{a:localApprox} holds and that 
  $\tau$  is a stoping time adapted to the filtration $\mathcal{F}_n
  = \sigma(X_0,X_1,X_2,\dots,X_n)$ with $\E \tau < \infty$.  Let $g$ be a function 
  of the path as described above. If $g(X)$ is measurable with 
  respect to $\mathcal{F}_\tau$ then we have the following
  result:
  \begin{align*}
    \| \text{Law}(g(X)) - \text{Law}(g(X^\epsilon)) \|_{TV} \leq 
    \epsilon \, \E \tau 
  \end{align*}
  if $\text{Law}(X_0)=\text{Law}(X_0^\epsilon)$.
\end{theorem}
An interesting example of such a function is the hitting time of a set 
$A$. In this case, $g(X)=\inf\{ n  \geq 0 : X_n \in A\}$.

\section{Understanding Through Coupling}
\label{sec:underst-thro-coupl}

The main point of this section is to give a path-wise perspective on
results of this flavor. We will use the following assumption which in
our setting is more natural than  Assumption~\ref{a:doeblin}. It can
be viewed as a ``cross-Doeblin'' condition.
\begin{assumption}\label{a:crossDoeblin}
    There exists a constant $\alpha\in (0,1)$ and  $\epsilon_0 \in (0,\alpha)$ so that
  \begin{align*}
  \|  \P_\epsilon(x,\ccdot)-\P(y,\ccdot)\|_{TV} \leq 1-\alpha
  \end{align*}
for     all $x,y \in\X$ and $\epsilon \in (0,\epsilon_0]$.
\end{assumption}

\begin{proposition}\label{prop:assumptionsEquiv}
  Assumptions \ref{a:doeblin} and \ref{a:localApprox} holding with
  parameters $\epsilon_0$ and $a$ implies  Assumptions \ref{a:localApprox}  and
  \ref{a:crossDoeblin}  hold with the same $\epsilon_0$ and
  $\alpha=a-\epsilon$. Similarly   assumptions  \ref{a:localApprox}  and
  \ref{a:crossDoeblin}   holding with
  parameters $\epsilon_0$ and $\alpha$ implies  Assumptions
  \ref{a:doeblin} and \ref{a:localApprox}  hold with the same $\epsilon_0$ and
  $a=\alpha-\epsilon$.
\end{proposition}
\begin{proof}[Proof of Proposition~\ref{prop:assumptionsEquiv}]
 If  Assumptions \ref{a:doeblin} and \ref{a:localApprox}  hold then
 \begin{align*}
   \|  \P_\epsilon(x,\ccdot)-\P(y,\ccdot)\|_{TV} &\leq \|
   \P_\epsilon(x,\ccdot)-\P(x,\ccdot)\|_{TV}
   +\|\P(x,\ccdot)-\P(y,\ccdot)\|_{TV} \\&\leq 1-a +\epsilon
 \end{align*}
which implies that  Assumption \ref{a:crossDoeblin} holds if
$\epsilon < a$ with $\alpha=a -\epsilon$. In the other
direction, 
\begin{align*}
\|  \P(x,\ccdot)-\P(y,\ccdot)\|_{TV}  &\leq  \|
  \P_\epsilon(x,\ccdot)-\P(y,\ccdot)\|_{TV}  + \|
                                        \P_\epsilon(x,\ccdot)-\P(x,\ccdot)\|_{TV}\\
  &\leq1- \alpha +\epsilon
\end{align*}
which completes the proof.
\end{proof}

The following Theorem~\ref{thm:avgProb} is one of the main results of
this note. It gives the existence of a coupling with certain
properties. With this result in hand, the proof of Theorem~\ref{cor:invAvg}
follows in a fashion inspired by the Coupling Time inequality of Aldous
used to bound mixing rates.
  
\begin{theorem}\label{thm:avgProb} Assume that Assumption~\ref{a:localApprox}  and
  Assumption~\ref{a:crossDoeblin} hold. Then for any pair of initial
  conditions $(X_0,X_0^\epsilon)$ there exists a coupling
  $(X_n,X_n^\epsilon)$ of the two chains so that
  \be
  \frac1n \sum_{k=0}^{n-1} \PP( X_k \not = X_k^\epsilon)  \le 
 \frac{\epsilon}{\alpha + \epsilon} +\frac{1-(1-\alpha-\epsilon)^n}{n(\alpha+\epsilon)}
  \Big( \PP\{X_0 \neq X_0^\epsilon\}  - \frac{\epsilon}{\alpha+\epsilon}\Big) \,.
  \ee
\label{thm:almostSure}
Furthermore  there exists a random constant $K=K(\alpha,\epsilon,X_0,X_0^\epsilon)$
  such that with probability one
  \begin{align*}
  \frac1n 
    \sum_{k=0}^{n-1}  \1_{\{X_k \neq X_k^\epsilon\}}\leq \frac{\epsilon}{\alpha+\epsilon} +  
    \frac{K}{n} + 2\sqrt{\frac{\log(n)}{n}}
  \end{align*}
for all $n \geq 0$. Finally we have the probabilistic bounds:
\begin{align*}
  \E \Big(\big[ \frac1n\sum_{k=0}^{n-1} \one_{\{X_k\neq X^\epsilon_k\}} -
  \frac{\epsilon}{\alpha+\epsilon}\big]^+\Big)^2  \leq   \frac{2}{n^2}\frac1{(\alpha+\epsilon)^2} + \frac2n\frac{\alpha^2}{(\alpha+\epsilon)^4}
\end{align*}
where $[x]^+=\max(x,0)$ and for any $\lambda >0$
\begin{align*}
   \PP\Big(  \frac1n \sum_{k=0}^{n-1}  \1_{\{X_k \neq X_k^\epsilon\}} 
  \geq \frac\epsilon{\alpha+\epsilon}+ \frac{\one_{X_0\neq X_0^\epsilon}}{n(\alpha+\epsilon)}+ \frac{\lambda}{\sqrt{n}} \Big)
  \leq e^{-\frac{(\alpha+\epsilon)^2}{2}\lambda^2}
\end{align*}
\end{theorem}

The proof of the first part of this theorem will be given in Section~\ref{sec:EBound}. 
The second part with the almost sure estimates and probabilistic bounds is
proved in Section~\ref{sec:ASBound}.


We have a second result that speaks to the distributions
of exit times and other path related quantities. The proof is given
in Section~\ref{sec:decoouling-time}.

\begin{theorem} \label{thm:deCoupling} Assume that Assumption~\ref{a:localApprox} holds. 
  Let $\tau$ be a stoping time adapted to the filtration $\mathcal{F}_n
  = \sigma(X_0,X_1,X_2,\dots,X_n)$ with $\E \tau < \infty$. 
Let $S_\epsilon=\inf\{ n: X_n \neq X_n^\epsilon\}$ where $(X_n,X_n^\epsilon)$
is  the coupled version of the process given in Theorem~\ref{thm:avgProb}. If we assume that
$X_0=X_0^\epsilon$ then 
\begin{align*}
  \PP( S_\epsilon \leq \tau ) \leq \epsilon\,\E\tau
\end{align*}
\end{theorem}

\subsection{Using the Coupling}
\label{sec:usingCoup}

We now use the results of the previous section to prove
Theorem~\ref{cor:invAvg} and Theorem~\ref{thm:Path}. In all cases the
idea is similar: use the fact that the two processes have been coupled
to agree often. For the statements in  Theorem~\ref{cor:invAvg}, we
use that they are equal for a controllable fraction of the
time. For Theorem~\ref{thm:Path}, we use that they are 
typically equal on a long interval of time if they agree initially.
\begin{proof}[Proof of Theorem~\ref{cor:invAvg} ]
We begin by proving the first statment.  Observe that 
    \begin{multline*}
   \frac12 \E\Big[ \frac1n \sum_{k=0}^{n-1} f(X_k) - \frac1n 
      \sum_{k=0}^{n-1} f(X_k^\epsilon)\Big] \leq \|f\|_\infty  \frac1n 
          \sum_{k=0}^{n-1} \PP(X_k \not = X_k^\epsilon)\\        
       \leq  \frac{\epsilon}{\alpha+\epsilon}+\frac{1-(1-\alpha-\epsilon)^n}{n(\alpha+\epsilon)^2}
    \big((\alpha+\epsilon) \PP\{X_0 \neq 
               X_0^\epsilon\} -\epsilon\big)
  \end{multline*}
where the last estimate comes from Theorem~\ref{thm:almostSure}. 
Since the above expression is true for any choice of couplings of the 
initial conditions 
$X_0$ and $X_0^\epsilon$, we are 
free to minimize over all such couplings. The Monge-Kantorovich 
Theorem states that $ \|\nu_1-\nu_2\|_{TV} =\inf \PP\{X_0 \neq 
               X_0^\epsilon\}$, where the infimum is taken over all 
               couplings  with marginals $\nu_1$ and $\nu_2$.  This
               produces the right hand side of the bound, while taking
               the supremum over all $f$ with $\|f\|_\infty \leq 1$
               produces the total variation norm on the left
               hand side. This completes the first statement.
The third statement follows from the estimate 
\begin{align*}
   \E\Big(\frac1n \sum_{k=0}^{n-1}    f(X_k)
     &- \frac1n \sum_{k=0}^{n-1}    f(X_k^\epsilon)
    \Big)^2 \leq  \E\Big(\frac1n \sum_{k=0}^{n-1}    [f(X_k)
     -   f(X_k^\epsilon)]\one_{\{X_k \neq X_k^\epsilon\}}  \Big)^2 \\&\leq 2\|f\|_\infty^2 \E\Big(\frac1n \sum_{k=0}^{n-1}  \one_{\{X_k \neq X_k^\epsilon\}} 
    \Big)^2 \\&\leq 4\|f\|_\infty^2 \Big[\frac{\epsilon^2}{(\alpha+\epsilon)^2}+ \E\Big(\big[\frac1n
                                                                       \sum_{k=0}^{n-1}
                \one_{\{X_k \neq X_k^\epsilon\}}
                -\frac\epsilon{\alpha+\epsilon}\big]^+ \Big)^2 \Big]
\end{align*}
and from Theorem~\ref{thm:almostSure}. The almost sure statment and  exponential estimate
follow similarly, also from Theorem~\ref{thm:almostSure}.
\end{proof}

The proof of Theorem~\ref{thm:Path} is very similar.
\begin{proof}[Proof of Theorem~\ref{thm:Path}] For any bounded
  function $f\colon \R \rightarrow \R$, we have
  \begin{align*}
   \E \big[f( g(X)) - f(g(X^\epsilon)) \big]&=  \E \big[f( g(X)) -
    f(g(X^\epsilon)) \big]\big[\one_{\{\tau < S_\epsilon\}}+
    \one_{\{\tau  \geq S_\epsilon\}}\big]\\&= \E \big[f( g(X)) -
    f(g(X^\epsilon)) \big] \one_{\{\tau  \geq S_\epsilon\}}\leq
                                             2\|f\|_\infty \P(\tau  \geq S_\epsilon)
  \end{align*}
Where $S_\epsilon$ is the decoupling time defined in the statement of 
Theorem~\ref{thm:deCoupling}.
The result now follows from Theorem~\ref{thm:deCoupling} on the
decoupling time $S_\epsilon$.
\end{proof}

\subsection{Construction of the Coupling}\label{sec:buildingTheCoupling}

Given any two probability measures $m_1$ and $m_2$ on $\X$, one can always
write them as a density relative to a common probability
measure $m$, namely $m=\frac12(m_1+m_2)$. If $m_i(dx) = f_i(x) m(dx)$
for $i=1,2$ for some $f_1,f_2 \in L^1(\X,m)$, we then define the
measures $m_1 \wedge m_2$  and $[m_1-m_2]^+$ by $(m_1 \wedge m_2)(dx)
= (f_1 \wedge f_2)(x) m(dx)$ and  $[m_1-m_2]^+(dx) = [f_1(x)
-f_2(x)]^+ m(dx)$ respectively. It is not hard to see that
\begin{align*}
  \|m_1 - m_s\|_{TV} = 1- (m_1\wedge m_2)(\X) = [m_1-m_2]^+(\X)  =  [m_2-m_1]^+(\X).
\end{align*}

We define the following measures on  $\X$
which will be used to construct our coupling. For any
$\xi=(\xi_1,\xi_2) \in \X \times \X$, we define
\begin{align*}
  Q_\epsilon(\xi, \ccdot) = &\frac{\P_\epsilon (\xi_1,\ccdot) \wedge
  \P(\xi_2,\ccdot)}{\rho_\epsilon(\xi)}\\ R_\epsilon(\xi,\ccdot) = \frac{[\P_\epsilon (\xi_1,\ccdot) -
  \P(\xi_2,\ccdot)]^+}{1-\rho_\epsilon(\xi)}&\qquad \widetilde R_\epsilon(\xi,\ccdot) = \frac{[
  \P(\xi_2,\ccdot) -\P_\epsilon (\xi_1,\ccdot)]^+}{1-\rho_\epsilon(\xi)} \,  ,
\end{align*}
where $\rho_\epsilon(\xi) = 1- \|\P_\epsilon( \xi_1,\ccdot) -
\P(\xi_2,\ccdot)\|_{TV}$. By the preceding observations these are
all probability measures on $\X$ for fixed $\xi \in \X \times \X$. Now define the following
transition kernels in $\X\times \X$ for $\xi=(\xi_1,\xi_2)$ and
$x=(x_1,x_2)$ in $\X\times \X$ by 
\begin{align*}
  \Q_\epsilon(\xi,dx) = \rho_\epsilon(\xi) Q_\epsilon(\xi, dx_1) \delta_{x_1}(dx_2)  +
  \big(1-\rho_\epsilon(\xi)\big) \big( R_\epsilon(\xi,dx_1) \times \widetilde R_\epsilon (\xi,dx_2)\big)\,.
\end{align*}
Observe that the marginals of $\Q_\epsilon(\xi,dx)$ are respectively
$\P_\epsilon (\xi_1,\ccdot)$  and $\P(\xi_2,\ccdot)$.

Notice that under Assumptions \ref{a:localApprox} and
\ref{a:crossDoeblin},  this construction has the following
properties. If $\chi=(\chi_1,\chi_2)$  is distributed  according to
$\Q_\epsilon(\xi,\ccdot)$ then $P(\chi_1 = \chi_2) = \rho_\epsilon(\xi)$ and 
\begin{align*}
\rho_\epsilon(\xi) \geq
  \begin{cases}
1-\epsilon \qquad & \text{if $\xi_1 = \xi_2$}\\
\alpha & \text{if $\xi_1 \not = \xi_2$}\\
      \end{cases}
\end{align*}

Letting $\chi_n= ( \chi_n^{(1)},\chi_n^{(2)})$  be the  Markov
chain on $\X \times \X$  defined by the transition density
$\Q_\epsilon$, we  define  the stochastic process $Z_n^\epsilon$ by
\begin{align*}
  Z_n^\epsilon =
  \begin{cases}
0 & \quad\text{if } \chi_n^{(1)} =     \chi_n^{(2)}\\
1 & \quad \text{if } \chi_n^{(1)} \not =     \chi_n^{(2)}\\
  \end{cases}
\end{align*}
While $Z_n^\epsilon$ is not Markovian, we can define the random quantities 
$\PP( Z_{n+1}^\epsilon =k  \,|\, Z_n^\epsilon=j)$
by $\E (\1\{Z_{n+1^\epsilon} =k\}  \,|\, Z_n^\epsilon=j)$. Now observe that with probability
one 
\begin{align*}
\PP( Z_{n+1}^\epsilon =0  \,|\, Z_n^\epsilon=0) \geq1- \epsilon \quad\text{and}\quad
\PP( Z_{n+1}^\epsilon =0  \,|\, Z_n^\epsilon=1) \geq \alpha\,.
\end{align*}
Let $Y_n$ be the Markov chain on $\{0,1\}$ with the transition
matrix
\begin{align}
P_\epsilon=\begin{pmatrix}
  1 - \epsilon & \epsilon \\
\alpha & 1- \alpha
\end{pmatrix}  \,.
\end{align}
Assuming that $\epsilon < 1-\alpha$, we have that  with probability one
\begin{align*}
  \PP( Z_{n+1}^\epsilon=0 \,|\,  Z_n^\epsilon=0) \geq&  \PP( Y_{n+1} =0  \,|\, Y_n =0) =
                                  1-\epsilon\\
  \PP( Z_{n+1}^\epsilon =0  \,|\,  Z_n^\epsilon=1) \geq&  \PP( Y_{n+1} =0  \,|\, Y_n =1) =\alpha\\
  \PP( Z_{n+1}^\epsilon =0  \,|\,  Z_n^\epsilon=0) \geq&  \PP( Y_{n+1} =0  \,|\, Y_n =1) = \alpha\,.
\end{align*}
Either directly from these estimates or  from the fact that they imply
that 
\begin{align*}
  \PP( Z_{n+1}^\epsilon \leq k  \,|\, Z_n^\epsilon \leq Y_n )  \geq   \PP( Y_{n+1} \leq k  \,|\, Z_n^\epsilon \leq Y_n )
\end{align*}
for all $k \geq 0$ and $n \geq 0$, which is the assumption of classical 
stochastic dominance theorems, it is clear that one can
construct a monotone coupling of the 
processes  $Y_n$ and $Z_n^\epsilon$.  That is, we can
construct copies of $Y_n$ and $Z_n^\epsilon$ 
on the same probability space such that
\begin{align}\label{eq:ZYordering}
  \PP(  Z_{n}^\epsilon \leq Y_n \text{ for all $n$}) =1 
\end{align}
provided $Z_0^\epsilon \leq Y_0$. In particular, this implies
that with probability one
\begin{align}\label{eq:ZYorderingImplied}
  \frac1n \sum_{k=0}^{n-1} \one_{\{X_k \neq X_k^\epsilon\}}  = \frac1n
  \sum_{k=0}^{n-1} \one_{\{ Z_{k}^\epsilon=1 \}}  \leq \frac1n
  \sum_{k=0}^{n-1}\one_{\{ Y_{k}=1 \}}  =  \frac1n
  \sum_{k=0}^{n-1}\phi( Y_{k})\,.
\end{align}
Hence to control the fraction of time $X_n$ and $X_n^\epsilon$ disagree
it is enough to bound the amount of time $Y_n=1$. The
analysis of this bounding chain is the topic of the next section.

\subsection{Analysis of Bounding Chain}\label{sec:Analysis1}
We now give the proofs of the estimates in
Theorem~\ref{cor:invAvg}. The basic idea is to use the fact that
$Z_n^\epsilon$ is stochastically dominated by $Y_n$ in the
sense of \eqref{eq:ZYordering} and \eqref{eq:ZYorderingImplied}, to
reduce all questions of interest to statements about the time that
$Y_n$ spends in state 1. Since $Y_n$ is a simple two
state Markov chain, the analysis is elementary and quite explicit.
\subsubsection{Control in Expectation}
\label{sec:EBound}
The Markov transition matrix of the bounding chain is 
\begin{align}
\P_\epsilon=\begin{pmatrix}
  1 - \epsilon & \epsilon \\
\alpha & 1- \alpha
\end{pmatrix}  \,.
\end{align}
It has generator $L_\epsilon=\P_\epsilon-I$ and unique stationary measure $\mu_\epsilon$ given by
\begin{align}
  \mu_\epsilon =
  \begin{pmatrix}
    \frac{\alpha}{\alpha +\epsilon}& \frac{\epsilon}{\alpha+\epsilon}
  \end{pmatrix}
\end{align}
and satisfies by definition $\mu_\epsilon L_\epsilon = 0$ and $\mu_\epsilon
\P_\epsilon = \mu_\epsilon$.

We define the following vectors
\begin{align*}
  \phi =
  \begin{pmatrix}
    0\\ 1
  \end{pmatrix},
\quad 
1 = \begin{pmatrix}
    1\\ 1
  \end{pmatrix},
\quad \bar \phi_\epsilon = \mu_\epsilon \phi 1 = \begin{pmatrix}
    \frac\epsilon{\alpha + \epsilon}\\  \frac\epsilon{\alpha + \epsilon}
  \end{pmatrix}
\quad \text{and}\quad \widetilde \phi_\epsilon = \phi- \bar \phi_\epsilon
=
\begin{pmatrix}
  \frac{-\epsilon}{\alpha+\epsilon}\\ \frac\alpha{\alpha+\epsilon}
\end{pmatrix}
\end{align*}

and further define $\psi_\epsilon$ as the solution to the equation
\begin{align}
  L_\epsilon \psi_\epsilon = - \widetilde \phi_\epsilon
\end{align}
It is straightforward to see that 
\begin{align}
  \psi_\epsilon = \sum_{k=0}^\infty \P_\epsilon^k \widetilde \phi_\epsilon
\end{align}
Observe that $w_\epsilon$, defined by
\begin{align*}
  w_\epsilon =
  \begin{pmatrix}
    -\dfrac\epsilon\alpha\\1
  \end{pmatrix}\,,
\end{align*}
satisfies $\P_\epsilon w_\epsilon = (1-\epsilon-\alpha)w_\epsilon$ and
hence $w_\epsilon$ is right-eigenvector with eigenvalue
$1-\epsilon-\alpha$. Since  $\widetilde \phi_\epsilon =  \frac{\alpha}{\alpha +\epsilon} w_\epsilon$,
we have that
\begin{align*}
  \psi_\epsilon = \Big( \frac{\alpha}{\alpha +\epsilon}\Big)\Big(
  \sum_{k=0}^\infty (1-\epsilon-\alpha)^k \Big) w_\epsilon =
  \frac{\alpha}{(\alpha +\epsilon)^2} w_\epsilon
\end{align*}
where we have again used the fact that $\epsilon < 1- \alpha$ so that
$1-\epsilon-\alpha \in (0,1)$.
Any initial distribution of $(X_0,X_0^\epsilon)$ induced an
initial distribution  $\nu$  for the $Y_n$ chain by $\nu(0) = \PP(
X_0=X_0^\epsilon)$ and $\nu(1) = \PP(
X_0\not=X_0^\epsilon)$.

Combining the above properties we have that
\begin{equation*}
  \label{eq:NotEqProb}
  \nu \P_\epsilon^n \psi_\epsilon - \nu \psi_\epsilon =  \sum_{k=0}^{n-1} \nu \P^k_\epsilon
  L_\epsilon \psi_\epsilon=  \sum_{k=0}^{n-1} \nu \P^k_\epsilon
 \phi - n \nu \bar \phi 
\end{equation*}
Rearranging this produces 
\begin{align*}
\frac1n \sum_{k=0}^{n-1} \nu \P_\epsilon^k \phi &= \frac{\epsilon}{\alpha+\epsilon} + \frac{\nu \P_\epsilon^n \psi_\epsilon - \nu \psi_\epsilon}{n} = \frac{\epsilon}{\alpha+\epsilon} + \frac{\alpha}{n(\alpha+\epsilon)^2} (1-(1-\alpha-\epsilon)^n) \nu w_\epsilon  \\
&= \frac{\epsilon}{\alpha+\epsilon} + \frac{1-(1-\alpha-\epsilon)^n}{n(\alpha+\epsilon)^2} \left(  \alpha \mathbf P(X_0 \ne X_0^\epsilon) -\epsilon(1- \mathbf P(X_0 \ne X_0^\epsilon)) \right) 
\end{align*}
Since $\frac1n \sum_{k=0}^{n-1} \mathbf P(X_k \neq X_k^\epsilon) \leq \frac1n
\sum_{k=0}^{n-1}  \nu \P_\epsilon^k \phi$, 
the preceding calculation proves
Theorem~\ref{thm:avgProb} after some algebra.

\subsubsection{Almost Sure Analysis and Variance}\label{sec:Analysis2}
\label{sec:ASBound}
Let $(X_n, X_n^\epsilon)$ be the coupled versions of the chains
constructed in the previous section and let $Z_n^\epsilon$ and $Y_n$ be the
associated processes on $\{0,1\}$ also constructed in the previous sections.
We now introduce slight abuse of
notation by allowing $\phi$, $\phi_\epsilon$, and $\psi_\epsilon$ to
denote the associated real valued functions on $\{0,1\}$. For example 
\begin{align*}
  \psi_\epsilon(y) =
  \begin{cases}
    - \frac{\epsilon}{(\alpha+\epsilon)^2} \quad& \text{if }y=0\\
     \frac{\alpha}{(\alpha+\epsilon)^2} \quad&\text{if } y=1
  \end{cases}
\end{align*}
Letting $\mathcal{F}_n$ be the filtration generated by
$(Y_0,Y_1,\cdots,Y_n)$,
define the Martingale increment
\begin{align*}
  I_n = \psi_\epsilon(Y_n)  -  \E(\psi_\epsilon(Y_n) | \mathcal{F}_{n-1})
\end{align*}
and the Martingale 
\begin{align*}
  M_n =\sum_{k=1}^n I_k
\end{align*}
with $M_0=0$.

Now since $\E(\psi_\epsilon(Y_n)- \psi_\epsilon(Y_{n-1})\,|\, \mathcal{F}_{n-1}) =
(L\psi_\epsilon)(Y_{n-1})$ we have
\begin{align*}
  \psi_\epsilon(Y_n) - \psi_\epsilon(Y_0) = \sum_{k=0}^{n-1} (L\psi_\epsilon)(Y_k) + M_n 
\end{align*}
and
\begin{align}\label{eq:asBound}
   \frac1n\sum_{k=0}^{n-1} \phi(Y_k) - \frac{\epsilon}{\alpha+\epsilon}&=  \frac1n\sum_{k=0}^{n-1} (L\psi_\epsilon)(Y_k) =
  \frac{\psi_\epsilon(Y_0) -\psi_\epsilon(Y_n) }{n} - \frac{M_n}{n} 
\end{align}
Next observe that 
\begin{align}
  \label{eq:phiBound}
- \frac{1}{\alpha+\epsilon}\1_{\{X_0 = X_0^\epsilon\}} \leq
  {\psi_\epsilon(Y_0) -\psi_\epsilon(Y_n) }& \leq
                                          \frac{1}{\alpha+\epsilon}\1_{\{X_0 \neq X_0^\epsilon\}}\,.
\end{align}

Since $|M_{n}- M_{n-1}| \leq \frac1{\alpha+\epsilon}$, we have by Azuma's inequality 
\begin{align*}
  \PP( |M_n| \geq \lambda \sqrt{n} ) \leq  2 e^{- \frac{(\alpha+\epsilon)^2}{2}\lambda^2}\,.
\end{align*}
Taking $\lambda= 2\sqrt{\log(n)}$ and using the Borel-Cantelli lemma
shows that there exists a random constant $K$ so that 
\begin{align*}
  |M_n| \leq K + 2 \sqrt{n \log(n)}
\end{align*}
for all $n \geq 0$.
Combining these estimates produces the following result, which shows that
under Assumptions \ref{a:doeblin} and \ref{a:localApprox},  for any
$\epsilon \in (0,\alpha]$ and initial conditions $X_0$ and
$X_0^\epsilon$ there exists a random,  postitive constant $K$, such that
      with probability one
      \begin{align*}
 -\frac{\one_{X_0= X_0^\epsilon}}{n(\alpha+\epsilon)} -
    \frac{K}{n} - 2\sqrt{\frac{\log(n)}{n}}\leq 
   \frac1n\sum_{k=0}^{n-1} \phi(Y_k) -\frac{\alpha}{\alpha+\epsilon} \leq 
    \frac{\one_{X_0\neq X_0^\epsilon}}{n(\alpha+\epsilon)} +
    \frac{K}{n} + 2\sqrt{\frac{\log(n)}{n}}         
      \end{align*}
for all $n \geq 0$. In addition for any $\lambda >0$ and $n >0$ one has
  \begin{align*}
    \PP\Big(  \frac1n\sum_{k=0}^{n-1} \phi(Y_k) -
    \frac{\epsilon}{\alpha+\epsilon}\geq  \frac{\one_{\{X_0\neq X_0^\epsilon\}}}{n(\alpha+\epsilon)}  + \frac{\lambda}{\sqrt{n}  } \Big)\leq e^{-\frac{(\alpha+\epsilon)^2}2\lambda^2}
  \end{align*}
Recalling that $\frac1n\sum_{k=0}^{n-1} \one_{\{X_0\neq
  X_0^\epsilon\}} \leq \frac1n\sum_{k=0}^{n-1} \phi(Y_k)$, produces
the first and last results given in Theorem~\ref{thm:almostSure}. To
see the second result we return to \eqref{eq:asBound}.  We use the fact that for $\epsilon \in (0,\alpha]$
\begin{align*}
 \E|M_n|^2 = \sum_{k=1}^n \E|I_k|^2 \leq  n  \frac{\alpha^2}{(\alpha+\epsilon)^4}
\end{align*}
to see that if $\epsilon \in(0,\alpha)$
then
\begin{align*}
 \E \Big| \frac1n\sum_{k=0}^{n-1} \phi(Y_k) -
  \frac{\epsilon}{\alpha+\epsilon}\Big|^2  \leq \frac{2 \E |  {\psi_\epsilon(Y_0) -\psi_\epsilon(Y_n) }|^2}{n^2} + \frac{2  \E|M_n |^2}{n^2}
 \leq 
  \frac{2}{n^2}\frac1{(\alpha+\epsilon)^2} + \frac2n\frac{\alpha^2}{(\alpha+\epsilon)^4}
\end{align*}
With this, the proof of last estimate from Theorem~\ref{thm:almostSure} is completed by
observing that
\begin{align*}
 \E \Big(\big[ \frac1n\sum_{k=0}^{n-1} \one_{\{X_k\neq X^\epsilon_k\}} -
  \frac{\epsilon}{\alpha+\epsilon}\big]^+\Big)^2\leq \E \Big(\big[ \frac1n\sum_{k=0}^{n-1} \phi(Y_k) -
  \frac{\epsilon}{\alpha+\epsilon}\big]^+\Big)^2\leq \E \Big| \frac1n\sum_{k=0}^{n-1} \phi(Y_k) -
  \frac{\epsilon}{\alpha+\epsilon}\Big|^2
\end{align*}

\subsection{Analysis of Decoupling Time}
\label{sec:decoouling-time}

Let $S_\epsilon =\inf(n: X_n\neq X_n^\epsilon)$ and let
$\sigma_\epsilon = \inf\{n : Y_n=1\}$ where $Y_n$ is the 0-1 Markov
process constructed in the previous section. Notice that the
construction is still possible if $\alpha=0$. In this case the state 1
is an absorbing state for the $Y_n$ chain, but the stochastic
ordering still holds.

Because of the stochastic
ordering $S_\epsilon \geq \sigma_\epsilon$. Hence for any stoping time
$\tau$,
\begin{align*}
  \PP( S_\epsilon \leq \tau) \leq \PP(\sigma_\epsilon \leq \tau)
\end{align*}

Now, 
\begin{align*}
   \PP(\sigma_\epsilon > \tau)= \sum_{k=0}^\infty  \PP(\sigma_\epsilon
  > k )\PP(\tau =k) =  \sum_{k=0}^\infty  (1-\epsilon)^k\PP(\tau =k) =
  \E (1-\epsilon)^\tau\, .
\end{align*}
Setting $\Lambda(\epsilon) =  \E (1-\epsilon)^\tau$, if we temporarily  assume
that $\tau \leq N$ almost surely for some constant $N$ then
$\Lambda(\epsilon)$ is an everywhere differentiable function. Expanding
around $0$ for $\epsilon>0$ and using the Lagrange remainder term produces 
\begin{align*}
  \Lambda(\epsilon) = 1 - \epsilon \E \tau + \tfrac12 c^2 \E ( \tau(\tau -1) )
  \geq 1 - \epsilon \E \tau 
\end{align*}
for some $c \in [0, \epsilon]$.  Since both the right and left hand
side are well defined for any $\epsilon \in (0,1)$, when the stopping time only
satisfies $\E\tau <\infty$ (rather than the almost sure bound $\tau <N$), we
conclude that 
\begin{align*}
 \PP(\sigma_\epsilon > \tau) \geq 1 - \epsilon \E \tau 
\end{align*}
in general. Since $ \PP(\sigma_\epsilon \leq \tau)  = 1 -
\PP(\sigma_\epsilon > \tau) $ the result is proven.

\section{Sharpness}\label{sec:Sharp}

Now we show that the total variation bound in Theorem~\ref{cor:invAvg} is tight by exhibiting a Markov chain satisfying the assumptions that achieves the bound. Let
\begin{align*}
 \P = \begin{pmatrix} 1-\beta & \beta \\ \beta & 1-\beta \end{pmatrix}
\end{align*}
for $\beta \le 1/2$. It is easy to verify by direct calculation that the invariant measure is $\mu = \begin{pmatrix} \frac12 & \frac12 \end{pmatrix}$ and $\P$ satisfies the Doeblin condition with $a = 2\beta$. $\P$ has eigenvectors
\begin{align*}
\phi_1 = (\tfrac12,\tfrac12), \quad \phi_2 = (-\tfrac12,\tfrac12)
\end{align*}
with eigenvalues $1$ and $1-2\beta$, respectively. Any possible starting measure $\nu$ can be expressed as $\nu_{\gamma} = (\gamma,1-\gamma)$ for some $\gamma \le \tfrac12$. Then $\|\nu_{\gamma} - \mu\|_{TV} = \frac{1}{2}\left( |\tfrac12-\gamma| + |\tfrac12-(1-\gamma)| \right) = \frac{1}{2}-\gamma$ when $\gamma < \tfrac12$; note if $\gamma > \tfrac12$, we get $\gamma -\tfrac12$. 


Consider the perturbation
\begin{align*} 
 \P_{\epsilon} = \begin{pmatrix} 1-(\beta-\epsilon) & \beta-\epsilon \\ \beta+\epsilon & 1-(\beta+\epsilon) \end{pmatrix},
\end{align*}
which satisfies $\sup_{x \in \mathcal X} \| \P_{\epsilon}(x,\cdot)- \P(x,\cdot)\|_{TV} = \epsilon$ and $\sup_{(x,y) \in \mathcal X \times \mathcal X} \|\P_\epsilon(x,\cdot)-\P(y,\cdot)\|_{TV} < 1-(2\beta-\epsilon) = 1-\alpha$. A diagonalization of $\P_\epsilon$ is
\begin{align*}
D_\epsilon &= Q_\epsilon^{-1} \P_\epsilon Q_\epsilon = \frac{1}{2a} \begin{pmatrix} \beta+\epsilon & \beta-\epsilon \\ -1 & 1 \end{pmatrix} \begin{pmatrix} 1 & 0 \\ 0 & 1-2\beta \end{pmatrix}  \begin{pmatrix} 1 & -(\beta-\epsilon) \\ 1 & \beta+\epsilon \end{pmatrix}
\end{align*}
(see \cite[pp 15-16]{lawler2006introduction} for example), so that
\begin{align*} 
\P^k_\epsilon &= \frac{1}{2\beta} \begin{pmatrix} (\beta+\epsilon)+(\beta-\epsilon)(1-2\beta)^k & (\beta-\epsilon)-(\beta-\epsilon)(1-2\beta)^k \\ (\beta+\epsilon) - (\beta+\epsilon) (1-2\beta)^k & (\beta-\epsilon) + (\beta+\epsilon)(1-2\beta)^k \end{pmatrix}.
\end{align*}
Therefore
\be
\nu_\gamma \P^k_\epsilon 
&= \tfrac{1}{2\beta} \begin{pmatrix} (\epsilon + \beta )+(\beta(2\gamma - 1) - \epsilon)(1-2\beta)^k & (\beta-\epsilon) + (\beta(1-2\gamma)+\epsilon) (1-2\beta)^k \end{pmatrix}
\ee
so 
\begin{align*}
\frac1n\sum_{k=0}^{n-1} &\nu_\gamma \P^k_\epsilon = \tfrac{1}{2\beta} \begin{pmatrix} (\epsilon + \beta )+(\beta(2\gamma - 1) - \epsilon)\frac{1-(1-2\beta)^n}{2\beta n} & (\beta-\epsilon) + (\beta(1-2\gamma)+\epsilon) \frac{1-(1-2\beta)^n}{2\beta n} \end{pmatrix} \\
&= \begin{pmatrix} (\frac{\epsilon}{a} + \frac{1}{2} )+(\| \mu - \nu_\gamma\|_{TV}  - \frac{\epsilon}{a})\frac{1-(1-a)^n}{a n} & (\frac{1}{2}-\frac{\epsilon}{a}) + (-\|\mu- \nu_\gamma\|_{TV} +\frac{\epsilon}{a}) \frac{1-(1-a)^n}{a n} \end{pmatrix}
\end{align*}
where in the last step we took $\gamma > \tfrac12$. Recalling that $\mu = \begin{pmatrix} \frac{1}{2}  & \frac{1}{2} \end{pmatrix}$ we get
\begin{align*}
\left\|\mu - \frac1n \sum_{k=0}^{n-1} \nu_\gamma \P^k_\epsilon \right\|_{TV} &= \frac{1}{2} \left| \frac{1}{2} - \left( (\frac{\epsilon}{a} + \frac{1}{2} )+( \| \mu - \nu_\gamma\|_{TV}  - \frac{\epsilon}{a})\frac{1-(1-a)^n}{a n} \right) \right| \\
&= \frac{\epsilon}{\alpha+\epsilon} + \big(\|\mu - \nu_\gamma\|_{TV} - \frac{\epsilon}{\alpha+\epsilon}\big) \frac{1-(1-\alpha-\epsilon)^n}{(\alpha+\epsilon) n} 
\end{align*}
 where $\alpha$ is the constant appearing in Assumption
 \ref{a:crossDoeblin}. For $\alpha>\epsilon$, Corollary
 \ref{cor:invAvg} gives precisely this expression if one takes $\nu_1 = \mu$ and $\nu_2 = \nu$.
Thus we conclude that the total variation bound in
Theorem~\ref{cor:invAvg} is sharp.

\section{Application to MCMC for Gaussian processes}\label{sec:MCMC}
In this section we consider an application to a simple Markov chain
Monte Carlo algorithm for sampling from the posterior distribution in
a Gaussian process model for Bayesian analysis of spatially indexed data. 
Gaussian processes are also
commonly employed in nonparametric regression. Our aim is to exhibit a
case where one can achieve $\epsilon \ll \alpha+\epsilon$, and
therefore an accurate approximation of $\P$ by $\P_\epsilon$ via
Theorem~\ref{cor:invAvg}, with large
computational advantage. We consider approximation of a matrix inverse
of the form \be \label{eq:LowRankInv} (I_n+c \Sigma)^{-1} \approx
(I_n+c \Lambda \Lambda')^{-1} = I-\Lambda (c^{-1} I_q +
\Lambda'\Lambda)^{-1} \Lambda' \ee with $\Lambda$ a $n \times q$
matrix with $n \ll q$ and $I_k$ is a $k \times k$ identity
matrix. This is a common approach to achieve computational
tractability for Gaussian process models, since it replaces a
non-parallel $\mathcal O(n^3)$ algorithm with a parallelizable
$\mathcal O(n^2 q)$ algorithm. We form $\Lambda$ using a partial
spectral decomposition. Algorithms for approximating partial spectral
decompositions without computing the full spectral decomposition, with
applications to Gaussian process models, are given in
\cite{banerjee2013efficient}. Our aim is to assess the accuracy in the
total variation metric of approximations to Markov transition kernels
$\P$ that result from utilizing approximations of the form
\eqref{eq:LowRankInv} to generate $\P_\epsilon$.

Consider a Gaussian process model with squared exponential (or
``radial basis'') kernel 
\begin{equation}
    \label{eq:GPmodel}
    \begin{aligned}
      z(w) &= x_3 f(w) + \epsilon, \quad \epsilon \sim N(0,x_3^2) \\
\cov(f(w_i),f(w_j)) &= x_2 \exp(-x_1 \|w_i - w_j\|^2_2).  
    \end{aligned}
\end{equation}
The
parameters of the model are
$x = (x_1,x_2,x_3) \in \mathbf R_+^3 = \X$, the positive orthant
in $\mathbf R^3$. The points $W = w_1,\ldots,w_n$ at which the process
is sampled are treated as fixed and known, and the observations of the
process at these points are denoted $z =
(z(w_1),\ldots,z(w_n))$. Bayesian inference on $x$ requires choice of
a prior distribution. A default choice is an inverse Gamma prior on
$x_3^2$ with parameters $\tfrac{a}2,\tfrac{b}2$ and density \be p(x_3^2 \mid
\tfrac{a}2,\tfrac{b}2) = \frac{b^a}{\Gamma(a)} (x_3^2)^{-\frac{a}{2}-1}
e^{-\frac{b}{2x_3^2}}.  \ee For $x_1$, it is common (see
e.g. \cite{lum2012spatial}) to discretize the parameter space for
$x_1$ to $m$ points. We also do this for $x_2$, which allows us to
numerically compute the transition matrix. Specifically, our prior has
\be (x_1,x_2) &\in \X_1 \times \X_2, \quad |\X_1| = |\X_2| = m, \ee
almost surely, where $|\X|$ is the cardinality of the finite set
$\X$. We place prior mass $m^{-2}$ on each atom, leading to the
unnormalized posterior \be p(x) \propto |x_3^2(I_n+x_2
\Sigma(x_1,W))|^{-1/2} e^{-\frac{1}{2x_3^2} z'(I+x_2
  \Sigma(x_1,W))^{-1} z} (x_3^2)^{-\frac{a}{2}-1}
e^{-\frac{b}{2x_3^2}}, \ee where $\Sigma(x_1,W)$ is a $n \times n$
symmetric, positive-definite matrix with entries
\begin{align*}
  \{\Sigma(x_1,W)\}_{ij} = e^{-x_1 \|w_i-w_j\|_2^2}\,.
\end{align*}

Integration over 
$x_3^2$ is available in closed form, leading to the likelihood for $z$
marginal of $x_3$
\be \label{eq:MLikGP}
L(z \mid x_1,x_2,W) \propto |I+x_2 \Sigma(x_1,W)|^{-\frac12} \{b+z'(I+x_2 \Sigma(x_1,W))^{-1} z \}^{-\frac{a+n}2}.
\ee
Because the priors on $x_1,x_2$ are discrete uniform on $\X_1, \X_2$, respectively, 
the posterior, which is the target distribution we want to sample, is proportional 
to \eqref{eq:MLikGP} at the support points $(x_1,x_2) \in \X_1 \times \X_2$. We 
assess properties of the kernel $\P$ defined by the update rule
\be
r(y_2,x_1) &= \mathbf P(X_2 = y_2 \mid z, W, X_1=x_1) = \frac{L(z \mid x_1,y_2,W)}{\sum_{y_2^* \in \mathcal X_2} L(z \mid x_1, y_2^*,W)} \\
s(y_1,x_2) &= \mathbf P(X_1 = y_1 \mid z, W, X_2=x_2) = \frac{L(z \mid y_1,x_2,W)}{\sum_{y_1^* \in \mathcal X_1} L(z \mid y_1^*, x_2,W)}, 
\ee
which has invariant measure the posterior,
and an approximating kernel $\P_\epsilon$ that uses the same two-step update rule, 
but substitutes a low-rank approximation $\Sigma_\epsilon(x_1,W) = 
\Lambda_\epsilon(x_1,W) \Lambda_\epsilon(x_1,W)'$ where $\Lambda_\epsilon(x_1,W)$ 
is a $n \times q_\epsilon$ matrix, with $q_\epsilon \le n$ as in \eqref{eq:LowRankInv}. 
This is variously known as ``predictive process'' or ``subset of regressors,'' 
and is a common strategy for scaling computation in these models 
(see \cite{banerjee2013efficient, banerjee2008gaussian} and references therein). 

Observe that
\be
|\P((x_1,x_2),(y_1,y_2))-&\P_\epsilon((x_1^*,x_2^*),(y_1,y_2))| = |\P((x_1,\ccdot),(y_1,y_2))-\P_\epsilon((x_1^*,\ccdot),(y_1,y_2))| \\
&= |r(y_2,x_1)s(y_1,y_2) - r_\epsilon(y_2,x_1^*)s_\epsilon(y_1,y_2)| \\
&\equiv \alpha(x_1,x_1^*,y_1,y_2)
\ee
so that
\be
\sup_{x_1,x_1^*} \|\P((x_1,\ccdot),\ccdot)-\P_\epsilon((x_1^*,\ccdot),\ccdot) \|_{TV} = \frac{1}{2} \sup_{x_1,x_1^*} \sum_{y_1} \sum_{y_2} \alpha(x_1,x_2^*,y_1,y_2),
\ee
and the value of $\alpha$ in Assumption \ref{a:crossDoeblin} and $\epsilon$ in Assumption \ref{a:localApprox} can be computed exactly. 

To evaluate the practical usefulness of $\P_\epsilon$, we take $n=1000$, 
$w_i = \frac{i}{n}$, then take 100 independent samples from \eqref{eq:GPmodel} 
with 
\be
x_2 &= 0.9, \quad x_3^2 = 0.2, \quad x_1 = -\frac{\log(0.01)}{0.45} 
\ee
Here, $x_1$ is chosen such that $\Sigma(x_1,W) = 0.01$ when 
\be
\|w_i-w_j\|_2^2 = 0.45 \max_{i,j} \|w_i-w_j\|_2^2. 
\ee
Setting parameters such that the spatial correlation decays to 0.01 (or, more
generally, some small value) at distances equal to a specified fraction of the 
range of sampling points is typical in spatial statistics \cite{banerjee2014hierarchical}. 
For construction of $\P_\epsilon$ and $\P$, we put $m=10$ and 
\be
\X_1 &= \left\{\frac{-\log(0.01)}{0.95},\frac{-\log(0.01)}{0.85},\ldots,\frac{-\log(0.01)}{0.05} \right\} \\
\X_2 &= \{0.5,0.6,\ldots,1.4\}. 
\ee
For each sample from \eqref{eq:GPmodel}, we compute $\alpha$ and $\epsilon$ for values 
of $q_\epsilon$ between $1$ and $\min \{q_\epsilon : \epsilon < 10^{-10}\}$. 

Figure \ref{fig:EpsilonQ} shows results. The vertical axis shows $\frac{\epsilon}{\alpha+\epsilon}$ as a function of $q_\epsilon$. Recall from Theorem~\ref{cor:invAvg}, for example, that $\epsilon \ll \epsilon + \alpha$ results in high accuracy. The numerical simulation suggests one can achieve $\epsilon \ll \alpha+\epsilon$ with $q_\epsilon \ll n$, indicating that large computational gains from utilizing $\P_\epsilon$ are achievable. Indeed, $q_\epsilon=30$ is enough to have $\frac{\epsilon}{\alpha+\epsilon} < 10^{-4}$ for all 100 replicate simulations, thus allowing inversion of a $1000 \times 1000$ matrix required to compute $\P$ to be replaced by inversion of a $30 \times 30$ matrix in computing $\P_\epsilon$ while achieving an accurate approximation. When $n$ is large, dimension reduction of this magnitude can have a very large computational benefit, suggesting the practical value of this commonly used strategy, and allowing the bounds in this paper to be immediately applied to MCMC for Gaussian process models when the prior on $(x_1,x_2)$ is discrete.


\begin{figure}[h]
\centering
 \includegraphics[width=0.9\textwidth]{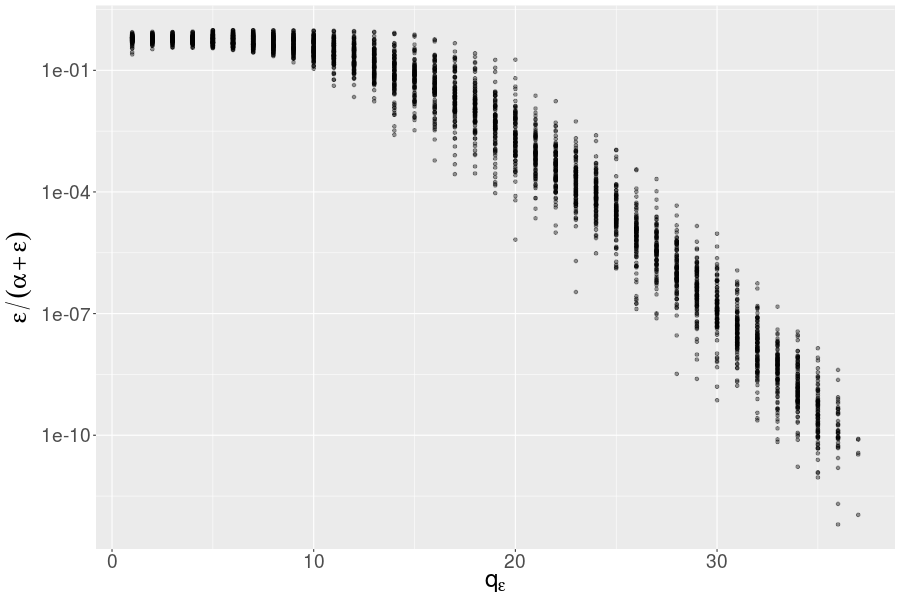}
 \caption{The value of $\frac{\epsilon}{\alpha+\epsilon}$, with $\alpha$ defined as in Assumption \ref{a:crossDoeblin}, as a function of the number of columns in the matrix $\Lambda(x_1,W)$ ($q_\epsilon$) for 100 independent samples from the model in \eqref{eq:GPmodel}. } \label{fig:EpsilonQ}
\end{figure}

\bibliography{closeChains}
\bibliographystyle{plain}

\appendix

\section{Proofs of basic results for $\P$ }
\label{sec:proofs-doebl-results}
 
Given any $f\colon \X \rightarrow \R$, consider the solution $\psi$ to 
the associated ``Poisson'' equation 
\begin{align}\label{eq:possion}
  L \psi = \mu f - f 
\end{align}
where $\mu$ is the unique invariant measure of $\P$, $\mu f= \int f 
d\mu$, and  $L$ is the generator of the Markov Chain defined by $L=\P-I$
where $I$ is the identity operator on $\X$. Recalling the definition 
of $|f|_*$ from \eqref{fstar}, we have the following result. 
\begin{lemma}\label{l:goodPoisson}
  If $f$ is bounded then there exits a unique solution $\psi$ of 
  \eqref{eq:possion}. Furthermore $|\psi|_\infty \leq 2|f|_*/a$. 
\end{lemma}
\begin{proof}[Proof of Lemma \ref{l:goodPoisson} ]
Define   
\begin{align*}
  \psi_n =\sum_{k=0}^{n} \P^k(f- \mu f ) 
\end{align*}
We will see that desired $\psi$ will be the limit of the 
$\psi_n$. Since the definition of $\psi$ does not change if $f$ is 
replaced by $f+\lambda$ for any constant $\lambda$, we are free to 
assume that $f$ is such that $|f|_\infty=|f|_*$. 

For any integers $n>m>0$ we have 
\begin{align*}
  |\psi_n - \psi_m|_\infty &\leq \sum_{k=m+1}^n |\P^k f - \mu f|_\infty 
  \leq 2 |f|_\infty \sum_{k=m+1}^n (1-a)^k \\&=2 (1-a)^{m+1}|f|_\infty \frac{1-(1-a)^{n-m}}{a}
\end{align*}
Hence $\psi_n$ is a Cauchy sequence  and the limit exists which we 
will call $\psi$. Now observe that 
\begin{align*}
  |\psi|_\infty \leq \sum_{k=0}^\infty |P^kf - \mu f|_\infty \leq 2 |f|_\infty 
  \sum_{k=0}^\infty (1-a)^k \leq \frac{2 |f|_\infty}a 
\end{align*}
and hence $\lim L\psi_n = L \psi$ since $L$ is a bounded operator. 
%
Now observe that since $L \mu f =0$
\begin{align*}
  L \psi_n = \sum_{k=0}^{n} (\P^{k+1}f - \P^kf) = \P^{n+1}f - f 
  \longrightarrow \mu f - f \quad\text{as}\quad n \rightarrow \infty \,. 
\end{align*}Recalling that  $|f|_\infty=|f|_*$, the proof is concluded. 
\end{proof}

Now 
\begin{align}
  \psi(X_n) - \psi(X_0) &= \sum_{k=1}^n\big( \psi(X_k) - \psi(X_{k-1}) 
  \big) =  \sum_{k=1}^n\E \big( \psi(X_k) - \psi(X_{k-1})\big|
  \mathcal{F}_{k-1} ) + M_n\notag \\
 &= \sum_{k=0}^{n-1} (L \psi)(X_k) + M_n\label{eq:averages}
\end{align}
where $\mathcal F_k$ is the $\sigma$-algebra generated by the random variables 
$(X_0,\dots,X_k)$ and $M_n$ defined by 
\begin{align*}
  M_n = \sum_{k=1}^n  \psi(X_k)  - \E\big( \psi(X_k) \big| \mathcal{F}_{k-1}\big) 
\end{align*}
is a Martingale with respect to $\mathcal{F}_n$. Now since $L\psi(X_k) 
= \mu f-  f(X_k)$ rearranging the above expression and 
dividing by $n$ produces 
\begin{align*}
  \mu f - \frac1n \sum_{k=0}^{n-1} f(X_k)  = \frac{\psi(X_{n}) -
  \psi(X_0)}n - \frac{M_n}n 
\end{align*}
and we have
\begin{align*}
 \E M_n^2  = \sum_{k=1}^n \E \Big( \psi(X_k)  - \E\big( \psi(X_k) 
  \big| \mathcal{F}_{k-1})\Big)^2.
\end{align*}

Letting $\mathcal A_{k-1} = \{Y: Y \text{ is } \mathcal F_{k-1}-\text{measurable }, \mathbf E Y^2 < \infty\}$, observe that
\be
\E \Big( \psi(X_k)  - \E\big( \psi(X_k) 
  \big| \mathcal{F}_{k-1})\Big)^2 = \inf_{Y \in \mathcal A_{k-1}} \E \Big( \psi(X_k)  - Y \Big)^2 \le \E \psi(X_k)^2
\ee
since the constant random variable $Y \equiv 0$ is in $\mathcal A_k$ for all $k$, giving
\be
 \E M_n^2  = \sum_{k=1}^n \E \Big( \psi(X_k)  - \E\big( \psi(X_k) 
  \big| \mathcal{F}_{k-1})\Big)^2 \le \sum_{k=1}^n \E \psi(X_k)^2 \le \frac{4 |f|_\infty^2}{a^2}n.
\ee

So we have that 
\begin{align*}
 \E \Big(\mu f - \frac1n \sum_{k=0}^{n-1} f(X_k)   \Big)^2 \leq 2 \frac{\E 
  \big(\psi(X_{n}) -
  \psi(X_0)\big)^2}{n^2} + 2\frac{\E M_n^2}{n^2} \leq 
  \frac{4 |f|_\infty^2}{a^2n}\big( 2+\frac{8}{n}\big) 
\end{align*}
Again recalling that  $|f|_\infty=|f|_*$, we get the first 
quoted result. 

Now to see the second result, observe that $|M_{k+1}- M_k| \leq 
4|f|_\infty/a$. Hence Azuma's inequality  implies that \\
%
\begin{align*}
   \PP( |M_n| \geq \lambda \sqrt{n}|f|_\infty ) \leq 2\exp\big(- \frac{a^2 \lambda^2}{32}\big)\,. 
\end{align*} and returning to the above calculation 
\begin{align*}
  \PP \Big(\big|  \mu f - \frac1n \sum_{k=0}^{n-1} f(X_k) \big| \geq 
  \frac{4}{n a}|f|_\infty + \frac{\lambda }{\sqrt{n}}   |f|_\infty \Big) 
  \leq 2\exp\big(- \frac{a^2 \lambda^2}{32}\big) 
\end{align*}

\begin{remark}[Closeness by Perturbation  Estimate]\label{rem: Closeness}
 We now briefly
  give a version of the closeness result starting from
  Assumption~\ref{a:doeblin}  and Assumption~\ref{a:localApprox}
  following the proof above.

 Let $\psi$ the solution to the Poisson equation \eqref{eq:possion}
introduced previously. We begin by considering the analog of \eqref{eq:averages}
for the epsilon chain. 
\begin{align*}
  \psi(X_n^\epsilon) - \psi(X_0^\epsilon) 
  &= \sum_{k=0}^{n-1} (L_\epsilon  \psi)(X_k^\epsilon) + M_n^\epsilon= \sum_{k=0}^{n-1}    (L  \psi)(X_k^\epsilon) + R_n^\epsilon+ M_n^\epsilon 
\end{align*}
where $L_\epsilon=\P_\epsilon-I$ is the generator associated to 
$\P_\epsilon$ and $M_n^\epsilon$ and $R_n^\epsilon$ are defined by 
\begin{align*}
  M_n^\epsilon = \sum_{k=1}^n  \psi(X_k^\epsilon)  - \E\big(
  \psi(X_k^\epsilon) \big| \mathcal{F^\epsilon}_{k-1}\big)\,, \quad 
  R_n^\epsilon =  \sum_{k=0}^{n-1}
    \big((\P-\P_\epsilon)  \psi\big)\big(X_k^\epsilon\big) 
\end{align*}
where $\mathcal{F}_k^\epsilon$ is the $\sigma$-algebra generated by 
$(X_0^\epsilon,\dots,X_k^\epsilon)$. Notice that we have used the fact 
that $L-L_\epsilon = \P-\P_\epsilon$. Using \eqref{eq:possion}, we 
obtain 
\begin{align}\label{eq:poisson}
 \mu f - \frac1n \sum_{k=0}^{n-1}    f(X_k^\epsilon) = \frac{\psi(X_n^\epsilon) - \psi(X_0^\epsilon)}{n} +\frac1n 
  R_n^\epsilon + \frac1n M_n^\epsilon 
\end{align}
Since the right hand side does not change if $f$ is replaced by 
$f+\lambda$, we can assume $|f|_\infty=|f|_*$. 
By Assumption~\ref{a:localApprox} followed by Lemma~\ref{l:goodPoisson}, 
we see that 
\begin{align*}
  \frac1n |R_n^\epsilon| \leq 2\epsilon |\psi|_\infty \leq \frac{4\epsilon}{a}
  |f|_\infty, \quad\frac1{n^2} \E |M_n^\epsilon|^2 \leq \frac{4 |f|_\infty^2}{a^2n}, \quad |M_{k}^\epsilon- M_{k-1}^\epsilon|\leq \frac{4|f|_\infty}{a} 
\end{align*}

From this we can quickly obtain estimates which are the analogue of
those in Theorem~\ref{cor:invAvg}. First observe that taking expectations
  of \eqref{eq:poisson}  and using these estimates produces 
  \begin{align*}
     \mu f - \E \frac1n \sum_{k=0}^{n-1}    f(X_k^\epsilon) \leq \frac{4}{an} |f|_\infty
   +  \frac{4\epsilon}{a} |f|_\infty
  \end{align*}
Similarly we have,
\begin{align*}
  \E\Big(\mu f - \frac1n \sum_{k=0}^{n-1}    f(X_k^\epsilon)  \Big)^2 
  \leq 3\frac{4|\psi|_\infty^2}{n^2}+ 3 |R_n|^2 + 3 \E |M_n^\epsilon|^2 \leq 
  \frac{3}{a^2}( 16\epsilon^2 + \frac{16}{n^2}) 
  |f|_\infty^2 + \frac{12}{a^2n}|f|_\infty^2 
\end{align*}
and as before using  Azuma's inequality 
\begin{align*}
  \PP\Big(\big| \mu f - \frac1n \sum_{k=0}^{n-1}    f(X_k^\epsilon)  \big|
  \geq \frac{4}{a}( \epsilon + \frac{1}{n})|f|_\infty  + \frac{\lambda }{\sqrt{n}}|f|_\infty 
  \Big)\leq 2\exp\big( - \frac{ a^2\lambda^2}{32}\big) 
\end{align*}
 \end{remark}

\end{document}